\documentclass[11pt]{article}

\usepackage{amsmath}
\usepackage{amsfonts}
\usepackage{amssymb}
\usepackage[all,2cell]{xy}
\UseTwocells

\newtheorem{proposition}{Proposition}

\newenvironment{proof}{{\bf Proof:}}{$\text{ }\blacksquare$}

\begin{document}

\title{Topological characterization of various types of $C^\infty$-rings}
\author{Dennis Borisov\\ Max-Planck Institute for Mathematics, Bonn, Germany\\ dennis.borisov@gmail.com}
\date{\today}
\maketitle

\begin{abstract} 
\noindent Topologies on algebraic and equational theories are used to define germ determined, near-point determined, and point determined $\mathcal C^\infty$-rings, without requiring them to be finitely generated. It is proved, that any $\mathbb R$-algebra morphism (without requiring continuity) into a near-point determined $\mathcal C^\infty$-ring is a $\mathcal C^\infty$-morphism (and hence continuous).\end{abstract}

\section{Introduction}

When using algebraic-geometric approach to deal with smooth manifolds and singular $C^\infty$-spaces, one works with $\mathcal C^\infty$-rings, i.e. rings having not only polynomial operations (as it is for the commutative rings), but all possible smooth functions as operations (\cite{Du81},\cite{GS03},\cite{Jo11},\cite{MR91},\cite{Sp10}). In this framework, it is very important to choose correctly the right subcategory of the category $\mathcal L$ of all $\mathcal C^\infty$-rings, so as to avoid unwanted anomalies.

This choice is usually made using topological or analytic properties of the rings $C^\infty(\mathbb R^n)$ for all $n\geq 0$ (\cite{Wh48},\cite{Du81},\cite{MR91}). In this paper we describe the systematic way to do so, and as a result we obtain the categories of germ determined, near-point determined, and point determined $\mathcal C^\infty$-rings, without requiring the rings to be finitely generated. We show that near-point determined $\mathcal C^\infty$-rings are very close to commutative $\mathbb R$-algebras, in that every $\mathbb R$-algebra morphism (without requiring continuity) into a near-point determined $\mathcal C^\infty$-ring is automatically a $\mathcal C^\infty$-morphism (compare \cite{Re75},\cite{GS03}).

\smallskip

Our approach is based on the notion of {\it a semi-topological theory}, which is a theory together with a topology, s.t. composition of operations is separately continuous in each variable. We define $3$ topologies on the theory of smooth functions $\mathcal C^\infty$:  basic open sets for each one of the topologies are obtained by fixing germs or finite jets or values of functions at finite sets of points. We show that closed ideals in these topologies are precisely the germ determined, near-point determined, and point determined ones respectively.

Since we define topology on the theory itself, and not directly on rings, we do not need to require the rings to be finitely generated. Just from comparing the topologies, we obtain a chain of full reflective subcategories $\mathcal E\subset\mathcal F\subset\mathcal G\subset\mathcal L$ of point determined, near-point determined, and germ determined $\mathcal C^\infty$-rings respectively.

\smallskip

Here is the structure of the paper: in section \ref{Theories} we define three topologies on the theory of smooth functions, and prove that each one of them makes $\mathcal C^\infty$ into a semi-topological theory. In section \ref{Algebras} we use the notion of a natural topology on an algebra to single out Hausdorff algebras. We apply this to the germ, jet, and point topologies to obtain the categories $\mathcal G, \mathcal F,\mathcal E$. Finally we prove that an $\mathbb R$-algebra morphism into $A\in\mathcal F$ is always a $\mathcal C^\infty$-morphism.

\smallskip

{\it Acknowledgements:} This work was done to clarify questions, that arose during the seminar on Derived Differential Geometry at Max-Planck-Institut f\"{u}r Mathematik in Bonn, during Summer 2011. The author would like to thank the participants, especially Christian Blohmann, Justin Noel, and Ansgar Schneider, for many fruitful discussions.

\section{Semi-topological theories of smooth functions}\label{Theories}

Recall that \underline{an algebraic theory} (\cite{La63}) is given by a small category $\mathbb T$, having all finite direct products, s.t. every object in $\mathbb T$ is a {\it finite} cartesian power of one chosen $T\in\mathbb T$. A $\mathbb T$-algebra in a category $\mathcal M$ is a product-preserving functor $\mathbb T\rightarrow\mathcal M$. We will denote the category of $\mathbb T$-algebras in $\mathcal M$ by $\mathbb T(\mathcal M)$ (morphisms between algebras are natural transformations).

 \underline{A morphism between algebraic theories} is a functor $\mathbb T\rightarrow\mathbb T'$, that preserves finite products and maps $T$ to $T'$. It is clear, that any such morphism induces a functor $\mathbb T'(\mathcal M)\rightarrow\mathbb T(\mathcal M)$.

\underline{The theory of smooth functions} $\mathcal C^\infty$ (\cite{Du81}) has $\{\mathbb R^n\}_{n\geq 0}$ as objects, and smooth maps between them as morphisms. We will follow established terminology: $\mathcal C^\infty$-algebras in $Set$ will be called \underline{$\mathcal C^\infty$-rings}, the category of such rings will be denoted by $\mathcal L$ (\cite{MR91}).\footnote{Note that, different from \cite{MR91}, we do not assume rings to be finitely generated.} It is well known (e.g. \cite{MR91}) that $\mathcal L$ contains the category of smooth manifolds as a full subcategory.

\smallskip

Let $\mathbb T$ be an algebraic theory. With some assumptions on $\mathcal M$ (\cite{BD80}), e.g. $\mathcal M$ being cartesian closed, the forgetful functor $\mathbb T(\mathcal M)\rightarrow\mathcal M$ has a left adjoint, i.e. there are free $\mathbb T$-algebras. For an $S\in Set$, the free $\mathbb T$-algebra in $\mathbb T(Set)$, generated by $S$, will be denoted by $\mathbb T(S)$. It is straightforward to check, that a free $\mathbb T$-algebra on $n$ generators ($n\in\mathbb Z_{\geq 0}$) is just $Hom(T^{\times^n},T)$. 

For example, a free $\mathcal C^\infty$-ring on $n$ generators is isomorphic to $C^\infty(\mathbb R^n)$ (e.g. \cite{MR91}). A free $\mathcal C^\infty$-ring on an infinite set of generators is a colimit (in $Set$) of a diagram of free finitely generated $\mathcal C^\infty$-rings and inclusions.

\smallskip

An algebraic theory consists of finitary operations, i.e. operations that have only finitely many inputs. In $\mathcal C^\infty$, for example, a typical operation is a smooth function $f:\mathbb R^n\rightarrow\mathbb R$. Even within a finite theory one encounters the need to consider operations with infinitely many inputs. For example, the free $\mathcal C^\infty$-ring, generated by a not necessarily finite set $S$, is the ring $C^\infty(\mathbb R^S)$ of smooth functions on $\mathbb R^S$.\footnote{Here, by a smooth function on we mean a function that factors through a projection $\mathbb R^S\rightarrow\mathbb R^n$, and a smooth function $\mathbb R^n\rightarrow\mathbb R$.}

To take care of such operations we use the notion of \underline{an equational theory} (\cite{Li66}),\footnote{What we call an equational theory here, is called a {\it varietal} equational theory in \cite{Li66}.} which is a category $\mathfrak T$, having all small direct products, s.t. every object in $\mathfrak T$ is a cartesian power of one chosen $T\in\mathfrak T$. \underline{A $\mathfrak T$-algebra} in a category $\mathcal M$ is a product preserving functor $\mathfrak T\rightarrow\mathcal M$.
We denote by $\mathfrak T(\mathcal M)$ the category of $\mathfrak T$-algebras in $\mathcal M$. Given two equational theories $\mathfrak T$, $\mathfrak T'$, a morphism $\mathfrak T\rightarrow\mathfrak T'$ is a product preserving functor, that maps $T$ to $T'$. For any $S\in Set$, the free $\mathfrak T$-algebra, generated by $S$, will be denoted by $\mathfrak T(S)$. It is easy to see (\cite{Li66}) that $\mathfrak T(S)\simeq Hom(T^{\times^{|S|}},T)$.

From any equational theory $\mathfrak T$, one can extract an algebraic sub-theory, which is the full subcategory $\mathbb T\subset\mathfrak T$, consisting of {\it finite} powers of $T$. Clearly this defines a functor
	\begin{equation}\label{InfiniteFinite}\xymatrix{\text{Equational theories }\ar[rr] && \text{ Algebraic theories.}}\end{equation}
It is known (\cite{Li66}), that the category of equational theories is equivalent to the category of monads on $Set$. Since every algebraic theory defines a monad, we have the following proposition.
\begin{proposition}\label{Completion} The functor (\ref{InfiniteFinite}) has a left adjoint, that will be denoted by  $\mathbb T\mapsto\underline{\mathbb T}$, s.t. $\forall S\in Set$
	\begin{equation} Hom_{\underline{\mathbb T}}(T^{\times^{|S|}},T)\simeq\mathbb T(S).\end{equation}
\end{proposition}
It follows immediately, that for any category $\mathcal M$, having all small direct products, the categories $\mathbb T(\mathcal M)$ and $\underline{\mathbb T}(\mathcal M)$ are naturally equivalent. 
So, by switching from $\mathcal C^\infty$ to $\underline{\mathcal C}^\infty$ we do not get anything new: $\underline{\mathcal C}^\infty$-algebras are just $\mathcal C^\infty$-rings. However, $\mathcal C^\infty$, in addition to algebraic structure, has rich topological and analytic properties. Extending them to $\underline{\mathcal C}^\infty$ does produce something new.

\smallskip

To deal with topology on $\mathcal C^\infty$, we need the notion of algebraic theories enriched in $Top$. Explicitly, \underline{a topological-algebraic theory} (\cite{BV73}) is a pair $(\mathbb T,\tau)$, where $\mathbb T$ is an algebraic theory, and $\tau$ is a topology on $Hom(T^{\times^m},T^{\times^n})$, $\forall m,n\geq 0$, s.t.\begin{itemize}
\item[1.] $Hom(T^{\times^m},T^{\times^n})\simeq Hom(T^{\times^m},T)^{\times^n}$ as topological spaces,
\item[2.] for any $l,m,n\geq 0$, the composition map 
	\begin{equation}\label{TopTheory}Hom(T^{\times^l},T^{\times^m})\times Hom(T^{\times^m},T^{\times^n})\rightarrow 
	Hom(T^{\times^l},T^{\times^n})\end{equation}
is continuous.\end{itemize}

On $\mathcal C^\infty$ there are several interesting topologies. There is the well known \underline{Whitney topology} (\cite{Wh48}), given by supremum norms on functions and their derivatives over compacts. 
We will define $3$ others, where basic open sets are given by fixing germs or finite jets or values of functions at finite sets of points. These topologies have nice extensions to $\underline{\mathcal C}^\infty$, which are not, what one would call topological equational theories, but something weaker.

\underline{A semi-topological equational theory} is an equational theory $\mathfrak T$, together with a topology $\tau$ on each $Hom(T^{\times^{|S_1|}},T^{\times^{|S_2|}})$, s.t.
	\begin{equation} Hom(T^{\times^{|S_1|}},T^{\times^{|S_2|}})\simeq Hom(T^{\times^{|S_1|}},T)^{\times^{|S_2|}}\end{equation}
as topological spaces, and the composition map 
	\begin{equation}\label{CompositionMap}Hom(T^{\times^{|S_1|}},T^{\times^{|S_2|}})\times Hom(T^{\times^{|S_2|}},T^{\times^{|S_3|}})\rightarrow Hom(T^{\times^{|S_1|}},T^{\times^{|S_3|}})
	\end{equation}
is separately continuous in each variable.

To define topology on $\underline{\mathcal C}^\infty$ we need to work with smooth functions on infinite dimensional real spaces. Let $S$ be a set, not necessarily finite. Let $\mathbb R^S$ be the $\mathbb R$-vector space of functions $S\rightarrow\mathbb R$. \underline{A smooth function} on $\mathbb R^S$ is a function that factors through projection on a finite dimensional summand $\mathbb R^S\rightarrow\mathbb R^n$, $n\in\mathbb Z_{\geq 0}$, and a smooth function $\mathbb R^n\rightarrow\mathbb R$. We denote by $C^\infty(\mathbb R^S)$ the $\mathcal C^\infty$-ring of smooth functions on $\mathbb R^S$. It is easy to see that $C^\infty(\mathbb R^S)$ is precisely the free $\mathcal C^\infty$-ring, generated by $S$. 

Two functions $f,g\in C^\infty(\mathbb R^S)$ \underline{have the same germ} at a point $p\in\mathbb R^S$, if there is $\mathbb R^n\subseteq\mathbb R^S$, s.t. $f$, $g$ factor through $\pi:\mathbb R^S\rightarrow\mathbb R^n$, $f|_{\mathbb R^n},g|_{\mathbb R^n}:\mathbb R^n\rightarrow\mathbb R$, and the germs of $f|_{\mathbb R^n},g|_{\mathbb R^n}$ at $\pi(p)$ are equal. It is clear that, in this case, for any finite dimensional subspace, containing $\mathbb R^n$, restrictions of $f,g$ have the same germ at the corresponding projection of $p$. Therefore, having the same germ at a given point is an equivalence relation, and, as usual, we will denote the germ of $f$ at $p$ by $f_p$. Clearly, if $S$ is finite, our notion of $f_p$ coincides with the standard one.

Two functions $f,g$ \underline{have the same $k$-jet} at $p$, if there is $\mathbb R^n\subseteq\mathbb R^S$, s.t. $f$, $g$ factor through $\pi:\mathbb R^S\rightarrow\mathbb R^n$, $f|_{\mathbb R^n},g|_{\mathbb R^n}:\mathbb R^n\rightarrow\mathbb R$, and the $k$-jets of $f|_{\mathbb R^n},g|_{\mathbb R^n}$ at $\pi(p)$ are equal. As with germs, it is easy to see that $\mathbb R^n$ can be enlarged, and still restrictions of $f,g$ will have the same $k$-jet at the corresponding projection of $p$. Therefore, having the same $k$-jet at $p$ is an equivalence relation, and we will denote the equivalence class of $f$ by $J_p^k(f)$. For a finite $S$, this notion coincides with the usual one.

\smallskip

For any $p\in\mathbb R^S$ there are several ideals of interest:\begin{itemize}
\item[1.] The ideal $\mathfrak m^g_p\subset C^\infty(\mathbb R^S)$ consists of functions, whose germ at $p$ is the same as that of $0$.
\item[2.] The ideal $\mathfrak m^k_p\subset C^\infty(\mathbb R^S)$ consists of functions, whose $k$-jet at $p$ is the same as that of $0$.
\item[3.] The ideal $\mathfrak m_p$ consists of functions, whose value at $p$ is $0$.\end{itemize}
It is obvious, that $\mathfrak m_p^k$ is indeed the $k$-th power of $\mathfrak m_p$. It is also easy to see, that $f_p=g_p$ if and only if $f-g\in\mathfrak m_p^g$, and similarly for equality of jets.

\smallskip

Now, for any $S\in Set$ we define:\begin{itemize}
\item[1.] A basis of the \underline{germ topology} is $\mathfrak U:=\{\emptyset\}\cup\{U_{\overline{p},f}\}$, where $\overline{p}=\{p_i\}$ is a finite set of points in $\mathbb R^S$, $f\in C^\infty(\mathbb R^S)$, and 
	\begin{equation} U_{\overline{p},f}:=\{g\in C^\infty(\mathbb R^S)\text{ s.t. } \forall i\text{ }f_{p_i}=g_{p_i}\}.\end{equation}
\item[2.] A basis of the \underline{jet topology} is $\mathfrak V:=\{\emptyset\}\cup\{V_{\overline{p},\overline{k},f}\}$, where $\overline{p}=\{p_i\}$ is a finite set of points in $\mathbb R^S$, $\overline{k}=\{k_i\}$ is a set of non-negative integers, one for each $p_i\in\overline{p}$, $f\in C^\infty(\mathbb R^S)$, and
	\begin{equation}\label{JetBasis} V_{\overline{p},\overline{k},f}:=\{g\in C^\infty(\mathbb R^S)\text{ s.t. }\forall i\text{ }J^{k_i}_{p_i}(f)=J^{k_i}_{p_i}(g)\}.\end{equation}
\item[3.] A basis of the \underline{point topology} is $\mathfrak W:=\{\emptyset\}\cup\{W_{\overline{p},f}\}$, where 
	\begin{equation}W_{\overline{p},f}:=\{g\in C^\infty(\mathbb R^S)\text{ s.t. }\forall i\text{ }g(p_i)=f(p_i)\}.\end{equation}\end{itemize}
First we prove that these are indeed bases of topologies.
\begin{proposition} Each one of the families $\mathfrak U,\mathfrak V,\mathfrak W$ is closed with respect to taking finite intersections.\end{proposition}
\begin{proof} Let $U_{\overline{p}_1,f_1},\ldots,U_{\overline{p}_n,f_n}\in\mathfrak U$. There are two cases. First, if all points in $\overline{p}_1,\ldots,\overline{p}_n$ are pairwise distinct, there is $k\in\mathbb Z_{\geq 0}$, s.t. each one of $f_1,\ldots,f_n$ factors through $\mathbb R^k$, and projections of all points to $\mathbb R^k$ are pairwise distinct. Then there is $f:\mathbb R^k\rightarrow\mathbb R$, s.t. $f_{p_{i,j}}=(f_i)_{p_{i,j}}$ $\forall i,j$, and therefore 
	\begin{equation}\underset{1\leq i\leq n}\bigcap U_{\overline{p}_i,f_i}=U_{\underset{1\leq i\leq n}\coprod\overline{p}_i,f}.\end{equation}
Suppose there are two equal points. Compare the corresponding functions at the point. If the germs are different, intersection is empty. If the germs are equal, one point can be eliminated. Similarly for the other two topologies.\end{proof}

\smallskip

Let \underline{$\tau^g$, $\tau^j$, $\tau^p$} be the germ, jet, and point topologies on $\underline{\mathcal C}^\infty$ respectively. 

\begin{proposition} Defined as above, $(\underline{\mathcal C}^\infty,\tau^g)$, $(\underline{\mathcal C}^\infty,\tau^j)$, $(\underline{\mathcal C}^\infty,\tau^p)$ are semi-topological equational theories.\end{proposition}
\begin{proof} Let $f\in C^\infty(\mathbb R^{S_1},\mathbb R^{S_2})$, $g\in C^\infty(\mathbb R^{S_2},\mathbb R^{S_3})$, and let $X\ni g\circ f$ be an open set with respect to any one of the $3$ topologies. Since $C^\infty(\mathbb R^{S_1},\mathbb R^{S_3})\simeq C^\infty(\mathbb R^{S_1})^{\times^{|S_3|}}$ as topological spaces, there are open sets $X_i\ni g_i\circ f$, $1\leq i\leq n$, s.t. $\pi_1^{-1}(X_1)\times\ldots\times\pi_n^{-1}(X_n)\subseteq X$. If we find open sets $Y_i\ni f$, $Z_i\ni g_i$, s.t. $g_i\circ Y_i\subseteq X_i$, $Z_i\circ f\subseteq X_i$, then clearly $(\underset{1\leq i\leq n}\bigcap\pi_i^{-1}(Z_i))\circ f\subseteq X$, $g\circ\underset{1\leq i\leq n}\bigcap Y_i\subseteq X$, and hence we can assume that $|S_3|=1$.

Let $f\in C^\infty(\mathbb R^{S_1},\mathbb R^{S_2})$, $g\in C^\infty(\mathbb R^{S_2})$, and consider $U_{\overline{p},h}$,where $h:=g\circ f$. By assumption, there is $m\geq 0$, s.t. $g$ factors through $\mathbb R^m$, let $f_1,\ldots,f_m$ be the corresponding projections of $f$. We claim that
	\begin{equation} g\circ\underset{1\leq i\leq m}\bigcap\pi_i^{-1}(U_{\overline{p},f_i})\subseteq U_{\overline{p},h},\quad 
	U_{f(\overline{p}),g}\circ f\subseteq U_{\overline{p},h}.\end{equation}
Both statements follow directly from the fact that composition of germs is a well defined germ. Similarly
	\begin{equation}g\circ\underset{1\leq i\leq m}\bigcap\pi_i^{-1}(V_{\overline{p},\overline{k},f_i})\subseteq V_{\overline{p},\overline{k},h},\quad 
	V_{f(\overline{p}),\overline{k},g}\circ f\subseteq V_{\overline{p},\overline{k},h},\end{equation}
and
	\begin{equation} g\circ\underset{1\leq i\leq m}\bigcap\pi_i^{-1}(W_{\overline{p},f_i})\subseteq W_{\overline{p},h},\quad 
	W_{f(\overline{p}),g}\circ f\subseteq W_{\overline{p},h}.\end{equation}\end{proof}

\smallskip

Moreover, a small alteration of the proof shows that $(\mathcal C^\infty,\tau^g)$, $(\mathcal C^\infty,\tau^j)$, $(\mathcal C^\infty,\tau^p)$ are topological algebraic theories, i.e. for the finite theories, composition is continuous, and not just separately continuous in each variable.


\section{Natural topologies on rings of smooth functions}\label{Algebras}

Presence of topology on an equational theory can be used to single out algebras, that have a particularly nice interaction with the topology. Let $(\mathfrak T,\tau)$ be a semi-topological equational theory, and let $A\in\mathfrak T(Set)$. Following \cite{KKM87}, we define \underline{natural topology} $\omega^\tau_A$ on $A$ to be the strongest topology, s.t. $\forall S\in Set$, $\forall\overline{a}\in A^{\times^{|S|}}$, the evaluation map
	\begin{equation}\label{EvaluationMap}\xymatrix{ev_{\overline{a}}:Hom(T^{\times^{|S|}},T)
	\ar[r]^{Id\times\overline{a}\quad} & Hom(T^{\times^{|S|}},T)\times A^{\times^{|S|}}
	\ar[r] & A}\end{equation}
is continuous.\footnote{Note that, different from \cite{KKM87}, we do not require $\omega_A$ to be compatible with the $\mathfrak T$-algebra structure on $A$ in any way. In \cite{KKM87}, in the case of $\mathcal C^\infty$, $(A,\omega_A)$ is required to be a locally convex, topological vector space.} As the following proposition shows, $\omega^\tau_A$ can be described explicitly, using a free resolution of $A$. The proof is straightforward (\cite{KKM87}).

\begin{proposition}\label{OrbitProperties} Let $(\mathfrak T,\tau)$ be a semi-topological equational theory.\begin{itemize} 
\item[1.] Let $A,B\in\mathfrak T(Set)$, and let $\omega^\tau_A,\omega^\tau_B$ be the natural topologies. Any morphism $\phi:A\rightarrow B$ in $\mathfrak T(Set)$ is continuous with respect to $\omega^\tau_A,\omega^\tau_B$. 
\item[2.] Let $\sim$ be a $\mathfrak T$-congruence on $A\in\mathfrak T(Set)$. Then $\omega^\tau_{A/\sim}=\omega^\tau_A/\sim$. 
\item[3.] For any $S\in Set$, restriction of $\tau$ to $Hom(T^{|S|},T)$ equals $\omega^\tau_{\mathfrak T(Set)}$.\end{itemize}\end{proposition}
We define \underline{$\mathfrak T_\tau(Set)\subseteq\mathfrak T(Set)$} to be the full subcategory, consisting of algebras, whose natural topology is Hausdorff. Such algebras will be called \underline{Hausdorff algebras}. The following proposition is straightforward (\cite{KKM87}).

\begin{proposition} Let $\mathfrak T'\subseteq\mathfrak T$ be a sub-theory, and suppose that $\mathfrak T'$ is dense in $\mathfrak T$ with respect to $\tau$. Let $A\in\mathfrak T(Set)$, $B\in\mathfrak T_\tau(Set)$, then any continuous $\mathfrak T'$-morphism $\phi:(A,\omega_A^\tau)\rightarrow(B,\omega_B^\tau)$ is a $\mathfrak T$-morphism.\end{proposition}

Let $\mathbb P$ be \underline{the theory of real polynomial functions}, i.e. objects of $\mathbb P$ are $\{\mathbb R^n\}_{n\geq 0}$, and morphisms are polynomial maps. Using multivariate Hermite interpolation, one shows that $\underline{\mathbb P}$ is dense in $\underline{\mathcal C}^\infty$ with respect to the jet topology. Therefore, any continuous $\mathbb R$-algebra morphism $(A,\omega_A^{\tau^j})\rightarrow(B,\omega_B^{\tau^j})$, with $\omega_B^{\tau^j}$ being Hausdorff, is automatically a $\mathcal C^\infty$-morphism. In Proposition \ref{FinalResult} we will see, that any $\mathbb R$-algebra morphism into such $B$ is continuous.

\smallskip

Now we would like to understand which $\mathcal C^\infty$-rings are Hausdorff with respect to the $3$ topologies that we have defined. First we note that every $\mathcal C^\infty$-ring is an abelian group. We will say that an equational theory $\mathfrak T$ \underline{contains the theory of groups}, if $T$ is a group object in $\mathfrak T$. So, $\underline{\mathcal C}^\infty$ contains the theory of groups. 

If $T\in\mathfrak T$ is a group object, every $\mathfrak T$-algebra is a group. If, in addition, $(\mathfrak T,\tau)$ is a semi-topological equational theory, separate continuity of (\ref{CompositionMap}) implies, that $\forall S\in Set$ all $\mathfrak T$-operations on $\mathfrak T(S)$ are continuous with respect to $\omega^\tau_{\mathfrak T(S)}$, in particular $\mathfrak T(S)$ is a topological group. It follows then, that any $\mathfrak T$-congruence on $\mathfrak T(S)$ is an open relation, and hence $\forall A\in\mathfrak T(Set)$ is a topological group with respect to $\omega^\tau_A$. Therefore, $A$ is Hausdorff, if and only if $A$ is a quotient of a free $\mathfrak T$-algebra by a closed normal subgroup.

So, a $\mathcal C^\infty$-ring is Hausdorff, if and only if it is quotient of a free $\mathcal C^\infty$-ring by a closed ideal. 
Hence, to understand Hausdorff $\mathcal C^\infty$-rings we need to understand closed ideals. Let $S\in Set$, following \cite{Du81}, \cite{MR91}, we define an ideal $\mathfrak I\subseteq C^\infty(\mathbb R^S)$ to be \underline{germ determined},  \underline{near-point determined}, or \underline{point determined}, if for any $f\in C^\infty(\mathbb R^S)$
	\begin{equation}\label{Germ} \forall p\in\mathbb R^S\text{ }\exists g\in\mathfrak I\text{ s.t. }f_p=g_p\text{ implies } f\in\mathfrak I,\end{equation}
	\begin{equation}\label{Jet} \forall p\in\mathbb R^S\text{ }\forall k\in\mathbb Z_{\geq 0}\text{ }\exists g\in\mathfrak I\text{ s.t. }J_p^k(f)=J_p^k(g)\text{ implies } f\in\mathfrak I,\end{equation}
	\begin{equation}\label{Point} \forall p\in\mathbb R^S\text{ }\exists g\in\mathfrak I\text{ s.t. }f(p)=g(p)\text{ implies } f\in\mathfrak I,\end{equation}
respectively. Note, that if $\exists g\in\mathfrak I$, s.t. $g(p)\neq 0$, then $\exists h\in\mathfrak I$, s.t. $h_p=1_p$, and hence such $p$ is irrelevant for checking conditions (\ref{Germ}), (\ref{Jet}), (\ref{Point}). In other words, everything depends on the points in $\mathbb R^S$, which are zeroes of $\mathfrak I$. In particular, if $\mathfrak I$ has no zeroes, it has to be all of $C^\infty(\mathbb R^S)$. One also sees that for a finite $S$, these definitions coincide with the ones in \cite{MR91}.

\begin{proposition} For any $S\in Set$, an ideal $\mathfrak I\subseteq C^\infty(\mathbb R^S)$ is closed with respect to the germ, jet, or point topology, if and only if it is germ determined, near-point determined, or point determined respectively.\end{proposition}
\begin{proof} Closure operator for the jet topology is as follows:  let $X\subseteq C^\infty(\mathbb R^S)$ and let $f\in C^\infty(\mathbb R^S)$, then $f\in\overline{X}$, if and only if for any finite decomposition $X=\underset{1\leq i\leq m}\coprod X_i$, there is at least one $X_i$, s.t. $\forall p\in\mathbb R^S$, $\forall k\in\mathbb Z_{\geq 0}$ there is $g\in X_i$, s.t. $J^k_p(g)=J^k_p(f)$. Indeed, suppose that $f\in\overline{X}$ for the jet topology. Let $X=\underset{1\leq i\leq m}\coprod X_i$ be a finite decomposition, and suppose for any $i$, there is $p_i\in\mathbb R^S$, and $k_i\in\mathbb Z_{\geq 0}$, s.t. $\nexists g\in X_i$ with $J^{k_i}_{p_i}(g)=J^{k_i}_{p_i}(f)$. Let $\overline{p}:=\{p_1,\ldots,p_m\}$, $\overline{k}:=\{k_1,\ldots,k_m\}$, and consider $V_{\overline{p},\overline{k},f}$ from (\ref{JetBasis}). Clearly $f\in V_{\overline{p},\overline{k},f}$, yet $V_{\overline{p},\overline{k},f}\cap X=\emptyset$. Therefore $f\notin\overline{X}$, contradiction.

Now suppose that $f\notin\overline{X}$, i.e. there are $\overline{p}\in(\mathbb R^S)^{\times^m}$, $\overline{k}\in(\mathbb Z_{\geq 0})^{\times^m}$ s.t. $V_{\overline{p},\overline{k},f}\cap X=\emptyset$. For any $Y\subseteq\{1,\ldots,m\}$ we define
	$$X_Y:=\{g\in X\text{ s.t. }\forall i\in Y\text{ }J_{p_i}^{k_i}(g)=J_{p_i}^{k_i}(f)\text{ and }\forall i\notin Y\text{ }J_{p_i}^{k_i}(g)\neq J_{p_i}^{k_i}(f)\}.$$
Clearly $X=\underset{Y}\coprod X_Y$, and $\forall Y\neq\{1,\ldots,m\}$ $\exists i$, s.t. $\nexists g\in X_Y$ with $J_{p_i}^{k_i}(g)=J_{p_i}^{k_i}(f)$. On the other hand, $X_{\{1,\ldots,m\}}=X\cap V_{\overline{p},\overline{k},f}=\emptyset$.

\smallskip

Let $\mathfrak I\subseteq C^\infty(\mathbb R^S)$ be an ideal, closed with respect to the jet topology. We claim that $\mathfrak I$ is near-point determined. Let $f$ be s.t. $\forall p\in\mathbb R^S$, $\forall k\in\mathbb Z_{\geq 0}$ $\exists g\in\mathfrak I$, s.t. $J^k_p(f)=J^k_p(g)$. We claim that $f\in\mathfrak I$. Suppose not. Then there is a finite decomposition $\mathfrak I=\underset{1\leq i\leq m}\coprod\mathfrak I_i$, s.t. for any $1\leq i\leq m$ $\exists p_i\in\mathbb R^S$ $\exists k_i\in\mathbb Z_{\geq 0}$, s.t. $\forall g\in\mathfrak I_i$ $J_{p_i}^{k_i}(g)\neq J_{p_i}^{k_i}(f)$. By assumption, $\forall i$ $\exists g_i\in\mathfrak I$, s.t. $J^{k_i}_{p_i}(g_i)=J^{k_i}_{p_i}(f)$. Since $\{f,g_1,\ldots,g_m\}$ is a finite set of functions, there is $\mathbb R^n\subseteq\mathbb R^S$, s.t. all of them factor through the projection to $\mathbb R^n$. Using partition of unity on $\mathbb R^n$, we can glue $g_i$'s into one $g\in\mathfrak I$, s.t. $\forall i$ $J^{k_i}_{p_i}(g)=J^{k_i}_{p_i}(f)$, which contradicts existence of the decomposition $\mathfrak I=\underset{1\leq i\leq m}\coprod\mathfrak I_i$ above.

Let $\mathfrak I$ be a near-point determined ideal in $C^\infty(\mathbb R^S)$. We claim that $\mathfrak I$ is closed in the jet topology. Let $f\in\overline{\mathfrak I}$, then $\forall p\in\mathbb R^S$, $\forall k\in\mathbb Z_{\geq 0}$, there is $g\in\mathfrak I$, s.t. $J^k_p(g)=J^k_p(f)$ (decomposition $\mathfrak I=\mathfrak I$). Since $\mathfrak I$ is near-point determined, $f\in\mathfrak I$.

The cases of germ or point topologies are proved in exactly the same manner.\end{proof}

\smallskip

Changing topologies on a given equational theory changes the corresponding categories of Hausdorff algebras. The following proposition shows that this change is well behaved. The proof is straightforward.
\begin{proposition} Let $(\mathfrak T,\tau)$, $(\mathfrak T,\tau')$ be semi-topological equational theories, s.t. $\tau\leq\tau'$. Then $\mathfrak T_\tau(Set)\subseteq\mathfrak T_{\tau'}(Set)$ as a full, reflective subcategory.\end{proposition}
Let \underline{$\mathcal E,\mathcal F,\mathcal G$} be the categories of Hausdorff $\mathcal C^\infty$-rings with respect to point, jet, and germ topologies respectively. It is immediate to notice that
	$$\text{point topology }<\text{ jet topology }<\text{ germ topology }<\text{ discrete topology,}$$
hence we have the corresponding sequence of full reflective subcategories
	\begin{equation}\mathcal E\subset \mathcal F\subset\mathcal G\subset\mathcal L.\end{equation}

\begin{proposition}\label{FinalResult} Let $A\in\mathcal L$, $B\in\mathcal F$. Any $\mathbb R$-algebra morphism $\phi:A\rightarrow B$ is a $\mathcal C^\infty$-morphism.\footnote{Note that, different from \cite{GS03}, we do not require the rings to be finitely generated.}\end{proposition}
\begin{proof} Suppose not. Let $\phi:A\rightarrow B$ be an $\mathbb R$-algebra morphism, that is not a $\mathcal C^\infty$-morphism. We can assume $A$ to be a finitely generated, free $\mathcal C^\infty$-ring. Indeed, there are $a_1,\ldots,a_n\in A$ and $f\in C^\infty(\mathbb R^n)$, s.t. 
	\begin{equation}\label{SupposeNot}\phi(f(a_1,\ldots,a_n))\neq f(\phi(a_1),\ldots,\phi(a_n)).\end{equation}
Let $\pi:C^\infty(\mathbb R^n)\rightarrow A$ be the $\mathcal C^\infty$-morphism , defined by $x_i\mapsto a_i$. We know that $\pi(f)=f(a_1,\ldots,a_n)$.
Therefore, if  $\phi(\pi(f))=f(\phi(\pi(x_1),\ldots,\pi(x_n)))$, then $\phi(f(a_1,\ldots,a_n))=f(\phi(a_1),\ldots,\phi(a_n))$,
i.e. if $\phi$ is not a $\mathcal C^\infty$-morphism, neither is $\phi\circ\pi$. Therefore, we can assume that $A=C^\infty(\mathbb R^n)$. 

So now we have an $\mathbb R$-algebra morphism $\phi:C^\infty(\mathbb R^n)\rightarrow B$, and a function $f\in C^\infty(\mathbb R^n)$, s.t. 
	\begin{equation}\phi(f)\neq f(\phi(x_1),\ldots,\phi(x_n)).\end{equation}
Take $S=B$ as the set of generators, and present $B=C^\infty(\mathbb R^S)/\mathfrak I$. In this way every $b\in B$ becomes a function on $\mathbb R^S$, that we will denote by $[b]$. Then, since $\mathfrak I$ is near-point determined, $\exists p\in C^\infty(\mathbb R^S)$, $\exists k\in\mathbb Z_{\geq 0}$, s.t. 
	\begin{equation}\label{Difference}[f(\phi(x_1),\ldots,\phi(x_n))]-[\phi(f)]\notin\mathfrak I+\mathfrak m_p^k.\end{equation}

Suppose that $\forall i$ $[\phi(x_i)](p)=0$. Taking Taylor formula for $f$ around $0$, we get $f=\rho+q_k\alpha$, where $\rho$ is a polynomial, $q_k$ is a homogeneous polynomial of degree $k$, and $\alpha\in C^\infty(\mathbb R^n)$. Then, writing $\overline{x}$ for $x_1,\ldots,x_n$, we have
	\begin{equation}[f(\phi(\overline{x}))]-(\rho([\phi(\overline{x})])+q_k([\phi(\overline{x})])\alpha([\phi(\overline{x})]))\in\mathfrak I,\end{equation}
and, since $\phi$ is an $\mathbb R$-algebra morphism, we have
	\begin{equation}[\phi(f)]-(\rho([\phi(\overline{x})])+q_k([\phi(\overline{x})])[\phi(\alpha)])\in\mathfrak I.\end{equation}
Therefore, (\ref{Difference}) is equivalent to
	\begin{equation} q_k([\phi(\overline{x})])(\alpha([\phi(\overline{x})])-[\phi(\alpha)])\notin\mathfrak I+\mathfrak m_p^k,\end{equation}
which is impossible, since $q_k([\phi(\overline{x})])\in\mathfrak m_p^k$.

Suppose now that $[\phi(x_i)]$ do not necessarily vanish at $p$. We would like to find a coordinate change on $\mathbb R^n$, s.t. images of the new coordinates do vanish at $p$, and there is a $g\in C^\infty(\mathbb R^n)$, satisfying (\ref{Difference}). Here is how it is done: evaluating $[\phi(x_i)]$'s at $p$, we get an $\mathbb R$-algebra morphism $\mathbb R[x_1,\ldots,x_n]\rightarrow\mathbb R$, i.e. a point in $\mathbb R^n$. Let $\mathbb R^n\rightarrow\mathbb R^n$ be the shift, that moves the origin to this point. Since this shift is a smooth map, we get a $\mathcal C^\infty$-morphism $\nu:C^\infty(\mathbb R^n)\rightarrow C^\infty(\mathbb R^n)$. Let $\{y_i\}$ be the new coordinate system after the shift. Since the shift is an algebraic map, $\{\nu(y_i)\}$ are polynomials in $x_i$'s, which we will denote by $\{\nu_i\}$. Let $\psi:=\phi\circ\nu$, and let $g:=\nu^{-1}(f)$. We claim that
	\begin{equation}\label{ReDifference}[g(\psi(\overline{y}))]-[\psi(g)]\notin\mathfrak I+\mathfrak m^k_p.\end{equation}
Indeed, $\psi(y_i)=\phi(\nu(y_i))=\phi(\nu_i(\overline{x}))=\nu_i(\phi(\overline{x}))$, and since $g(\nu_1,\ldots,\nu_n)=\nu(g)=f$, we have that $g(\psi(\overline{y}))=f(\phi(\overline{x}))$. On the other hand $\psi(g)=\phi(\nu(g))=\phi(f)$, and (\ref{ReDifference}) becomes (\ref{Difference}). Clearly $\forall i$ $[\psi(y_i)](p)=0$.\end{proof}

\end{document}